\documentclass[reqno,oneside]{amsart}

\pdfoutput=1

\makeatletter 
\def\@table{table}
\long\def\@makecaption#1#2{
\ifx\@captype\@table
 \begin{center}{\footnotesize #1}\\{\footnotesize\sc #2}\end{center}
 \vskip\abovecaptionskip\relax
\else
 \vskip\abovecaptionskip\relax
 \sbox\@tempboxa{{\footnotesize {#1.}~~ #2}}
 \hbox to\hsize{\footnotesize\box\@tempboxa\hfil}%
 \vskip\belowcaptionskip
\fi
}
\makeatother

\theoremstyle{plain}
\newtheorem{theorem}{Theorem}
\newtheorem{lemma}[theorem]{Lemma}

\theoremstyle{definition}
\newtheorem{remark}[theorem]{Remark}

\numberwithin{theorem}{section}

\usepackage{cite,booktabs,graphicx}

\newcommand{\?}{\;\!}

\title[]{A Chernoff-type Lower Bound for the Gaussian $\boldsymbol{Q}$-function}
\author[]{Fran\c{c}ois D. C\^ot\'e, Ioannis N. Psaromiligkos, and Warren J. Gross} 
\thanks{March 2012.}
\thanks{This research was supported in part by the Natural Sciences and Engineering Research Council of Canada (PGSD3-392045-2010) and by the Fonds qu\'eb\'ecois de la recherche sur la nature et les technologies (PR-127199).}
\thanks{The authors are with the Department of Electrical and Computer Engineering, McGill University, Montreal, QC H3A~2A7, Canada. (email: francois.cote@mail.mcgill.ca, yannis@ece.mcgill.ca, warren.gross@mcgill.ca).}

\begin{document}

\maketitle

\thispagestyle{empty}

\begin{abstract}

A lower bound for the Gaussian $Q$-function is presented in the form of a single exponential function with parametric order and weight.  We prove the lower bound by introducing two functions, one related to the $Q$-function and the other similarly related to the exponential function, and by obtaining inequalities that indicate the sign of the difference of the two functions. 
 
\medskip


\end{abstract}

\bigskip

\section{Introduction}

The Gaussian $Q$-function represents the probability that a standard normal random variable exceeds a value $x$ and is defined by
\[ Q(x) = \frac{1}{\sqrt{2\pi}}\int\limits_{x}^{\infty}e^{-\frac{t^2}{2}} \, dt \text{,} \quad x \in \mathbb{R} \text{.} \]
Because of the prevalence of normal random variables, the $Q$-function is arguably one of the most important integrals encountered in applied mathematics, statistics, and engineering. Unfortunately, the function is difficult to handle mathematically as it is an nonelementary integral, that is, it cannot be expressed as a finite composition of simple functions, such as algebraic, exponential, logarithmic, and rational functions.  For this reason, the development of approximations and bounds for the $Q$-function is a topic of continued research dating back to Laplace. 

In this note, we present a tight lower bound for the \makebox{$Q$-function} in the form of a single exponential function or, more precisely, in the form of a Gaussian function
\[ \alpha \? e^{-\beta x^2} \text{,} \]
where the weight $\alpha$ and the order $\beta$ are nonnegative constants and the argument $x$ is the same as in the $Q$-function.  Bounds in this form are usually referred to as Chernoff-type after the Chernoff bound~\cite{Chernoff1952}.  We choose to focus our attention on expressions of  this form because they easily accommodate analytical manipulations.  This is especially true, for example, in the analysis of communications systems, where these expressions are used to simplify various operations, like exponentiation and expectation, that must be applied to the $Q$-function to calculate the error rate of signals transmitted through fading channels~\cite{Alouini2004}.

It is known~\cite{Chang2011} that the tightest possible Chernoff-type upper bound for the $Q$-function is obtained for $\alpha = \beta = 1/2$, but a similar limit has not been established for the lower bound.  To the best of our knowledge, our bound is the tightest known Chernoff-type lower bound.

Chernoff-type lower bounds for the $Q$-function have received rather scant attention until recently.  Other than our bound, the only provable lower bounds are due to Chang et al.~\cite{Chang2011}, Fu and Kam~\cite{Fu2011}, and Chiani et al.~\cite{Chiani2003}.  Other lower bounds containing Gaussian functions are available (see \cite{Chiani2003, Tate1953, Wozencraft1965, Gordon1941, Birnbaum1942, Boyd1959, Abreu2009a, Abreu2009} for example), but these bounds are more complicated since they also involve other functions or a sum of multiple terms, and thus they are not of the Chernoff-type. 

The remainder of this note is devoted to establishing our bound.

\section{Main Result}  

\begin{theorem}\label{Thm:main}
For all real $x$ and any $\kappa \geq 1$, the following lower bound holds for the Gaussian $Q$-function:
\[ Q(x) \geq  \left(\frac{e^{(\pi(\kappa-1)+2)^{-1}}}{2\kappa} \sqrt{\frac{1}{\pi}(\kappa-1)(\pi(\kappa-1)+2)} \, \right) \! e^{-\frac{\kappa x^2}{2}} \text{.} \]
\end{theorem}

\subsection{Proof of Theorem~\ref{Thm:main}}Let 
\begin{equation}
g(x,\kappa) = \left( \frac{e^{(\pi(\kappa-1)+2)^{-1}}}{2\kappa}\sqrt{\frac{1}{\pi}(\kappa-1)(\pi(\kappa-1)+2)} \, \right) \! e^{-\frac{\kappa x^2}{2}} \text{.}  
\end{equation}
We begin by observing that $g(x,\kappa)$ is an even function of $x$ and that $Q(x)$ is a decreasing function of $x$.  Therefore, if the bound in Theorem~\ref{Thm:main} holds for $x \geq 0$, then it holds for all $x$.  So we restrict our attention to $x \geq 0$.  We also note that the bound holds trivially for $\kappa = 1$ because $g(x,1) = 0 < Q(x)$ for all $x \geq 0$.  Therefore, it remains to show that the bound holds for $x \geq 0$ and $\kappa > 1$. 

Let us define the functions
\begin{equation}
r(x,\kappa) = \sqrt{2\pi} \? g(x,\kappa) \? e^{\frac{x^2}{2}} \text{,} \quad x \geq 0 \text{,} \quad \kappa > 1 
\end{equation}
and
\begin{equation}
R(x) = \sqrt{2\pi} \? Q(x) \? e^{\frac{x^2}{2}} \text{,} \quad x \geq 0 \text{.}
\end{equation}
Consider the difference $r(x,\kappa)-R(x)$ denoted by $f(x,\kappa)$.  To prove Theorem~\ref{Thm:main}, we show that $f(x,\kappa) \leq 0$ for $x \geq 0$ and $\kappa > 1$.  The following lemmas provide relations involving $r(x,\kappa)$ and $R(x)$ to help determine the sign of $f(x,\kappa)$.

\begin{lemma}\label{Lem:f}
Suppose $x \geq 0$ and let $\kappa > 1$.  There exist unique points $x_1$ and $x_2$ that satisfy ${0 < x_1 < 1/\sqrt{\kappa-1} < x_2}$, such that the relation
\[ \kappa x \? r(x,\kappa) \geq 1 \]
holds if and only if $x \in [x_1,x_2]$.  Furthermore, the relation holds with equality if and only if $x = x_1$ or $x = x_2$.
\end{lemma}
\begin{proof}
We first rewrite the given relation into an equivalent form, considering $x \geq 0$ and $\kappa > 1$:
\begin{equation}\label{EQ:CLemf}
x^2(1-\kappa) \? e^{x^2(1-\kappa)} \leq 2(\pi(1-\kappa)-2)^{-1} \? e^{2(\pi(1-\kappa)-2)^{-1}} \text{.}
\end{equation}
Note that since $\kappa > 1$ we have $x^2(1-\kappa) \leq 0$ and $2(\pi(1-\kappa)-2)^{-1} < 0$.  The main idea behind the proof is to identify the necessary and sufficient conditions on $x$ that ensure the validity of (\ref{EQ:CLemf}).

We now define the function $h(w) = we^w$ for $w \leq 0$ and consider the inequality $h(w) \leq z$ for $-e^{-1} < z < 0$. We compute the derivative $\frac{d}{dw} \, h(w) = e^w+we^w$ and observe that $\frac{d}{dw} \, h(w) < 0$ for $w < -1$ and $\frac{d}{dw} \, h(w) > 0$ for $w > -1$.  In other words, $h(w)$ is decreasing for $x < -1$ and increasing for $x > -1$.  Thus $h(w)$ has a unique minimum $h(-1) = -e^{-1}$ at $w = -1$.  Moreover, $h(w)$ tends to $0$ as $w \to -\infty$ and $h(0) = 0$.  This tells us that $h(w) \leq z$ holds for some values of $w < 0$ as long as $-e^{-1} < z < 0$.  In this case, $h(w) \leq z$ holds if and only if the values of $w$ are between two endpoints, where one endpoint, say $w_1$, is greater than $-1$ and the other, say $w_2$, is less than $-1$.  Also, $h(w) = z$ holds if and only if $w$ is equal to either $w_1$ or $w_2$.  

Setting $z = 2(\pi(1-\kappa)-2)^{-1} \? e^{2(\pi(1-\kappa)-2)^{-1}}$, we see that the corresponding $w_1$ and $w_2$ must exist.  The reason is that here $z$ has the same form as $h(w)$, and thus $-e^{-1} < z < 0$. Setting $w = x^2(1-\kappa)$, we see that there is a unique positive value of $x$, say $x_1$, that satisfies $w = w_1$, and similarly there is a unique positive value of $x$, say $x_2$, that satisfies $w = w_2$.  The condition $w \in [w_2,w_1]$ is equivalent to $x \in [x_1,x_2]$, and $w_2 < -1 < w_1$ is equivalent to $x_1 < 1/\sqrt{\kappa-1} < x_2$.  It follows that (\ref{EQ:CLemf}), and thus the relation in the lemma, holds if and only if $x \in [x_1,x_2]$, where $0 < x_1 < 1/\sqrt{\kappa-1} < x_2$, and it holds with equality if and only if  $x = x_1$ or $x = x_2$.   
\end{proof}

\begin{remark}
The exact value of $x_1$ can be obtained from (\ref{EQ:CLemf}) as the solution to $x^2(1-\kappa) = 2(\pi(1-\kappa)-2)^{-1}$ for $x$.  This solution is less than $1/\sqrt{\kappa-1}$ and is given by
\begin{equation}\label{EQ:x_1}
x_1 = \frac{\sqrt{2}}{\sqrt{(\kappa-1)(\pi(\kappa - 1) +2)}} \text{.}
\end{equation}
The exact value of $x_2$ is not needed in the proof of Theorem~\ref{Thm:main}; the fact that it exists is enough for our purposes.
\end{remark}

\begin{lemma}\label{Lem:R}
Suppose $x \geq 0$ and let $\kappa > 1$.  The relation 
\[ \kappa x \? R(x) \geq 1 \]
holds if $x \geq x_1$, where $x_1$ is as in Lemma~\ref{Lem:f}.
\end{lemma}

\begin{proof}
A result of Boyd~\cite{Boyd1959} (see also \cite[p.~179,~(6)]{Mitrinovic1970}) says that, for $x \geq 0$,
\begin{equation}
R(x) \geq \frac{\pi}{(\pi-1)x+\sqrt{x^2+2\pi}} \text{.} 
\end{equation}
Therefore, a sufficient condition for the relation in the lemma to hold is that
\begin{equation}\label{EQ:CLemR}
\frac{\pi \kappa x}{(\pi-1)x+\sqrt{x^2+2\pi}} \geq  1 \text{.} 
\end{equation}
To prove the lemma, we show that (\ref{EQ:CLemR}) holds whenever $x \geq x_1$.  Considering $x \geq 0$ and $\kappa > 1$, we see that (\ref{EQ:CLemR}) translates into
\begin{equation}
x \geq \frac{\sqrt{2}}{\sqrt{(\kappa-1)(\pi(\kappa-1)+2)}} \text{.}
\end{equation}
Looking at (\ref{EQ:x_1}), this condition is equivalent to $x \geq x_1$.   \end{proof}

We are now ready to show that $f(x,\kappa) \leq 0$ for $x \geq 0$ and $\kappa > 1$.  We consider three cases:

\medskip

\begin{flushleft}

\textbf{Case 1:} $x \geq x_2$

By Lemma~\ref{Lem:f}, we know that $r(x,\kappa) \leq 1/(\kappa x)$ when $x \geq x_2$, and by Lemma~\ref{Lem:R}, we know that $R(x) \geq 1/(\kappa x)$ when $x \geq x_2 > x_1$.  Therefore, $f(x,\kappa) \leq 0$ for $x \geq x_2$.
 
\medskip

\textbf{Case 2:} $x \in [0,x_1]$

Let us examine $f(x,\kappa)$ at $x = x_1$.  By Lemma~\ref{Lem:f}, we know that $r(x_1,\kappa) = 1/(\kappa x_1)$, and by Lemma~\ref{Lem:R}, we know that $R(x_1) \geq 1/(\kappa x_1)$.  This means that $f(x_1,\kappa) \leq 0$.  Now we obtain
\begin{equation}\label{EQ:df}
\frac{\partial}{\partial x} \, f(x,\kappa) = x \? f(x,\kappa)+1-\kappa x \? r(x,\kappa) \text{.}
\end{equation}
Note by Lemma~\ref{Lem:f} that $\kappa x \? r(x,\kappa) < 1$ when $x \in [0,x_1)$.  Thus $\frac{\partial}{\partial x} \, f(x,\kappa) >  x \? f(x,\kappa)$.  To show that $f(x,\kappa) \leq 0$ when $x \in [0,x_1)$, we suppose, on the contrary, that $f(x,\kappa) > 0$ for some point in $[0,x_1)$.  Then as soon as $f(x,\kappa) > 0$, we have $\frac{\partial}{\partial x} \, f(x,\kappa) > 0$.  So $f(x,\kappa)$ must be increasing between that point and $x_1$.  This gives a contradiction since we know that $f(x_1,\kappa) \leq 0$.  Therefore, $f(x,\kappa) \leq 0$ for $x \in [0,x_1]$. 

\medskip

\textbf{Case 3:} $x \in [x_1,x_2]$

Recall from the previous case that $f(x_1,\kappa) \leq 0$.  By Lemma~\ref{Lem:f}, we know that $\kappa x \? r(x,\kappa) \geq 1$ when $x \in [x_1,x_2]$.  Thus, referring to~(\ref{EQ:df}), we have $\frac{\partial}{\partial x} \, f(x,\kappa) \leq x \? f(x,\kappa)$.  This means that $f(x,\kappa)$ must remain nonpositive for $x \in [x_1,x_2]$.

\medskip

\end{flushleft}

We have thus established our main result.


\begin{thebibliography}{13}

\bibitem{Chernoff1952}
H.~Chernoff, ``{A measure of asymptotic efficiency for tests of a hypothesis based on the sum of observations},'' \emph{Ann. Math. Statist.}, vol.~23, no.~4, pp. 493--507, Dec. 1952.

\bibitem{Alouini2004}
M.~K. Simon and M.-S. Alouini, \emph{{Digital Communication over Fading Channels}}, 2nd~ed.\hskip 1em plus 0.5em minus 0.4em\relax Hoboken, NJ: John Wiley \& Sons, 2004.

\bibitem{Chang2011}
S.-H. Chang, P.~C. Cosman, and L.~B. Milstein, ``{Chernoff-type bounds for the Gaussian error function},'' \emph{IEEE Trans. Commun.}, vol.~59, no.~11, pp. 2939--2944, Nov. 2011.

\bibitem{Fu2011}
H. Fu and P.-Y. Kam, ``{Exponential-type bounds on the first-order Marcum $Q$-function},'' \emph{in Proc. IEEE-GLOBECOM}, Houston, TX, Dec. 2011.

\bibitem{Chiani2003}
M.~Chiani, D.~Dardari, and M.~K. Simon, ``{New exponential bounds and approximations for the computation of error probability in fading channels},'' \emph{IEEE Trans. Wireless Commun.}, vol.~2, no.~4, pp. 840--845, July 2003.

\bibitem{Tate1953}
R.~F. Tate, ``{On a double inequality of the normal distribution},'' \emph{Ann. Math. Statist.}, vol.~24, no.~1, pp. 132--134, Mar. 1953.

\bibitem{Wozencraft1965}
J.~M. Wozencraft and I.~M. Jacobs, \emph{{Principles of Communications Engineering}}.\hskip 1em plus 0.5em minus 0.4em\relax New York, NY: John Wiley \& Sons, 1965.

\bibitem{Gordon1941}
R.~D. Gordon, ``{Values of Mills' ratio of area to bounding ordinate and of the normal probability integral for large values of the argument},'' \emph{Ann. Math. Statist.}, vol.~12, no.~3, pp. 364--366, Sep. 1941.

\bibitem{Birnbaum1942}
Z.~W. Birnbaum, ``{An inequality for Mills' ratio},'' \emph{Ann. Math. Statist.}, vol.~13, no.~2, pp. 245--246, Jun. 1942.

\bibitem{Boyd1959}
A.~V. Boyd, ``{Inequalities for Mills' ratio},'' \emph{Rep. Statist. Appl. Res. Un. Japan. Sci. Engrs.}, vol.~6, pp. 44--46, 1959.

\bibitem{Abreu2009a}
G.~T.~F. de~Abreu, ``{Supertight algebraic bounds on the Gaussian \makebox{$Q$-function}},'' in \emph{Proc. IEEE-ACSSC}, Pacific Grove, CA, Nov. 2009, pp. 948--951.

\bibitem{Abreu2009}
------, ``{Jensen-Cotes upper and lower bounds on the Gaussian \makebox{$Q$-function} and related functions},'' \emph{IEEE Trans. Commun.}, vol.~57, no.~11, pp. 3328--3338, Nov. 2009.

\bibitem{Mills1926}
J.~P. Mills, ``{Table of the ratio: area to bounding ordinate, for any portion of normal curve},'' \emph{Biometrika}, vol.~18, no. 3--4, pp.  395--400, Nov. 1926.

\bibitem{Mitrinovic1970}
D.~S. Mitrinovi\'c, \emph{{Analytic Inequalities}}.\hskip 1em plus 0.5em minus 0.4em\relax New York, NY: Springer-Verlag, 1970.

\end{thebibliography}
\end{document}